\documentclass[12pt]{article}
\title{Type-definable NIP fields are Artin-Schreier closed}
\author{Will Johnson}

\usepackage{amsmath, amssymb, amsthm}    	
\usepackage{fullpage} 	
\usepackage{amscd}
\usepackage{hyperref}
\usepackage[all]{xy}
\usepackage{centernot}
\usepackage{titlefoot}

\DeclareMathOperator*{\forkindep}{\raise0.2ex\hbox{\ooalign{\hidewidth$\vert$\hidewidth\cr\raise-0.9ex\hbox{$\smile$}}}}

\newcommand{\ACF}{\operatorname{ACF}}

\newcommand{\characteristic}{\operatorname{char}}

\newcommand{\Aut}{\operatorname{Aut}}

\newcommand{\dcl}{\operatorname{dcl}}

\newtheorem{theorem}{Theorem}[section] 
\newtheorem{lemma}[theorem]{Lemma}

\newtheorem{corollary}[theorem]{Corollary}
\newtheorem{fact}[theorem]{Fact}

\newtheorem*{theorem-star}{Theorem}
\newtheorem*{theorem-1}{Main Theorem}
\newtheorem*{lemma-star}{Theorem}
\newtheorem*{conjecture-star}{Conjecture}

\theoremstyle{definition}

\newtheorem{definition}[theorem]{Definition}
\newtheorem{example}[theorem]{Example}

\newtheorem{remark}[theorem]{Remark}

\newtheorem*{acknowledgment}{Acknowledgments}

\newcommand{\Aa}{\mathbb{A}}

\newcommand{\Zz}{\mathbb{Z}}

\newcommand{\Gg}{\mathbb{G}}

\newcommand{\Mm}{\mathbb{M}}
\newcommand{\Ff}{\mathbb{F}}

\begin{document}
\maketitle\unmarkedfntext{
  \emph{2020 Mathematical Subject Classification}: 03C45, 03C60.

  \emph{Key words and phrases}: NIP fields, type-definable fields.
  }

\begin{abstract}
  Let $K$ be a type-definable infinite field in an NIP theory.  If $K$
  has characteristic $p > 0$, then $K$ is Artin-Schreier closed (it
  has no Artin-Schreier extensions).  As a consequence, $p$ does not
  divide the degree of any finite separable extension of $K$.  This
  generalizes a theorem of Kaplan, Scanlon, and Wagner.
\end{abstract}

\section{Introduction}

An \emph{Artin-Schreier extension} is a non-trivial field extension of
the form $K(\alpha)/K$ where $\characteristic(K) = p$ and $\alpha^p -
\alpha \in K$.  If $K$ has characteristic $p$, then the Artin-Schreier
extensions are precisely the cyclic extensions of degree $p$.  We say
that a field $K$ of characteristic $p$ is \emph{Artin-Schreier closed}
if it has no Artin-Schreier extensions, or equivalently, the
Artin-Schreier map $\wp(x) = x^p - x$ is surjective.

Let $K$ be an infinite NIP field of characteristic $p$.  Then $K$ is
Artin-Schreier closed by a theorem of Kaplan, Scanlon, and Wagner
\cite[Theorem 4.4]{NIPfields}.  Here, we generalize this result to
type-definable NIP fields, answering a question asked in
\cite[Section~4]{NIPfields} and further studied in \cite[Question~1.14]{kaplan-shelah}.
\begin{theorem-1}
  Let $K$ be a type-definable field in an appropriately-saturated NIP
  structure.  If $K$ is infinite and has positive characteristic, then
  $K$ is Artin-Schreier closed.
\end{theorem-1}
The method of \cite[Corollary~4.5]{NIPfields} then gives the
following.
\begin{corollary}
  Let $K$ be a type-definable infinite field in an NIP structure.
  Then $p$ does not divide the degree of any finite separable
  extension of $K$.
\end{corollary}
\subsection{Outline}
We give two proofs of the main theorem.  Both proofs build off the
following theorem of Kaplan and Shelah, which is a variant of the
Baldwin-Saxl lemma:
\begin{fact}[{\cite[Corollary~3.3]{kaplan-shelah}}] \label{KS-fact}
  Let $\Mm$ be a highly saturated NIP structure.  Let $(K,\cdot)$ be a
  type-definable group, type-definable over $A \subseteq \Mm$.  Let
  $(G_i : i \in \Zz)$ be an $A$-indiscernible sequence of
  type-definable subgroups.  Then $G_0 \supseteq \bigcap_{i \ne 0}
  G_i$.
\end{fact}
In Section~\ref{ks-apply-sec} we use this to prove
Lemma~\ref{ks-apply} about counterexamples to the main theorem.  Then
the two proofs diverge.  In Section~\ref{alg-groups-version}, we give
a short proof of the main theorem relying on some machinery from algebraic groups, closely following Kaplan, Scanlon, and Wagner's original proof \cite{NIPfields}.  In
Section~\ref{elementary-version}, we give a more elementary and
explicit proof using additive polynomials.  The relation between the
two proofs is discussed in Section~\ref{comparison}, which also
provides conceptual motivation for the explicit proof of
Section~\ref{elementary-version}.

\begin{remark}
  Alternatively, the proof in Section~\ref{alg-groups-version} can be
  made explicit using calculations by Bays \cite[Appendix~A]{n-dep-2}.
\end{remark}

\subsection{Notation and conventions}
We use $ab$ to denote concatenation of tuples, and $a \cdot b$ to
denote multiplication.  We reserve the symbol $\Mm$ for monster
models, i.e., structures that are $\kappa$-saturated and strongly
$\kappa$-homogeneous for some cardinal $\kappa$ bigger than any
cardinals we care about.

Unless specified otherwise, $p$ is a prime and
all fields have characteristic $p$.  The prime field of size $p$
is written $\Ff_p$.  If $K$ is a field then $K^{alg}$ and $K^{sep}$
denote the algebraic and separable closures.  The perfect hull
$\bigcap_n K^{p^n}$ is denoted $K^{p^\infty}$.  The Artin-Schreier
polynomial $x^p - x$ is written $\wp(x)$.  If $V$ is a $K$-variety
(usually an algebraic group), and $L/K$ is an extension, then $V(L)$
denotes the set of $L$-points of $V$.  We denote the additive
algebraic group by $\Aa$ rather than the usual $\Gg_a$, to avoid
confusion with the groups $G_{\bar{a}}$ of Definition~\ref{ga-def}.
Then $\Aa(L) = (L,+)$.

\section{Application of NIP} \label{ks-apply-sec}
If $K$ is a type-definable field in some monster model $\Mm$, and $L$
is a finite extension, then we can interpret $L$ as a type-definable
field in $\Mm$, by identifying it with $K^d$ for $d = [L : K]$.  Under
this interpretation, $K$ is a relatively definable subset of $L$,
meaning that both $K$ and $L \setminus K$ are type-definable subsets
of $L$.
\begin{lemma}\label{ks-apply}
  Let $K$ be a type-definable infinite field of positive
  characteristic in an NIP structure $\Mm$.  Let $L/K$ be a
  non-trivial Artin-Schreier extension.  Let $B \subseteq \Mm$ be a
  small set of parameters over which $K$ and $L$ are type-definable.
  Let $(c_i : i \in \Zz)$ be a $B$-indiscernible sequence in
  $K^\times$.  Then there is $N$ such that
  \begin{equation}
    (c_0 \cdot \wp(L \setminus K)) \cap \bigcap_{i = 1}^N (c_i \cdot
    \wp(K)) \cap \bigcap_{i = -N}^{-1} (c_i \cdot \wp(K)) =
    \varnothing. \label{n-target}
  \end{equation}
\end{lemma}
\begin{proof}
  Because $K$ and $L \setminus K$ are type-definable, so are the two
  sets
  \begin{align*}
  A &:= \wp(K) \cap K = \wp(K) \\
  U &:= \wp(L \setminus K) \cap K
  \end{align*}
  The sets $A$ and $U$ are disjoint---if $u \in K$ and $v \in L
  \setminus K$ and $\wp(u) = \wp(v)$, then $u-v \in \Ff_p$, so $v \in
  K + \Ff_p = K$, a contradiction.  The set $A$ is a type-definable
  subgroup of $K$, so we can apply Fact~\ref{KS-fact} to the
  sequence $(c_i \cdot A : i \in \Zz)$, concluding that $c_0 \cdot A
  \supseteq \bigcap_{i \ne 0} (c_i \cdot A)$.  Because $A$ and $U$ are
  disjoint, $(c_0 \cdot U) \cap \bigcap_{i \ne 0} (c_i \cdot A) =
  \varnothing$.  By compactness, there is a finite $N > 0$ such that
  \begin{equation*}
    (c_0 \cdot U) \cap \bigcap_{i = 1}^N (c_i \cdot A) \cap \bigcap_{i =
      -N}^{-1} (c_i \cdot A) = \varnothing.
  \end{equation*}
  This implies Equation (\ref{n-target}).
\end{proof}
In the following sections, our strategy will be to build an
indiscernible sequence $(c_i : i \in \Zz)$ for which Equation
(\ref{n-target}) fails.

\section{Proof via algebraic groups} \label{alg-groups-version}
Recall that $\Aa$ denotes the additive group, as an algebraic group.
The following definition is from \cite{NIPfields}:
\begin{definition} \label{ga-def}
  If $K$ is a field of characteristic $p$ and $\bar{a} \in K^n$, then
  $G_{\bar{a}}$ is the algebraic subgroup of $\Aa^n$ defined by
  \begin{equation*}
    G_{\bar{a}}(L) = \{(t,x_1,\ldots,x_n) \in L^{n+1} : t = a_1 \cdot
    \wp(x_1) = a_2 \cdot \wp(x_2) = \cdots = a_n \cdot \wp(x_n)\},
  \end{equation*}
  for any $L/K$.
\end{definition}
\begin{fact}[{\cite[Corollary~2.9]{NIPfields}}] \label{they-are-additive-group}
  If $K$ is perfect and $\bar{a} \in K^n$ is algebraically independent
  over $\Ff_p$, then $G_{\bar{a}}$ is isomorphic to $\Aa$ as an
  algebraic group over $K$.
\end{fact}
In fact, Hempel shows that the conclusion holds as long as
$(1/a_1,\ldots,1/a_n)$ is $\Ff_p$-linearly independent
\cite[Corollary~5.4]{n-dep-1}.  An explicit isomorphism between
$G_{\bar{a}}$ and $\Aa$ is given by Bays \cite[Appendix~A]{n-dep-2}.
\begin{fact}\label{come-on}
  Let $K$ be a field of characteristic $p$.  Suppose $\bar{a} \in K^n$
  is algebraically independent over $\Ff_p$.  Let $\bar{a}'$ be a
  subtuple of length $n-1$, obtained by removing one element.  Then
  the projection $\pi : G_{\bar{a}} \to G_{\bar{a}'}$ is isomorphic
  (in the category of algebraic groups over $K$) to the Artin-Schreier
  endomorphism $\wp : \Aa \to \Aa$.
\end{fact}
Fact~\ref{come-on} is implicit in the proof of
\cite[Theorem~4.4]{NIPfields}.
\begin{theorem-1}
  Let $\Mm$ be a monster model of an NIP theory.  Let $K$ be a
  type-definable infinite field of characteristic $p > 0$.  Then $K$
  is Artin-Schreier closed.
\end{theorem-1}
\begin{proof}
  Suppose for the sake of contradiction that there is an
  Artin-Schreier extension $L/K$.  Let $B$ be a small set of
  parameters over which $K$ and $L$ are type-definable.  As in the
  proof of \cite[Theorem~4.4]{NIPfields}, let $k$ be the perfect hull
  $K^{p^\infty} = \bigcap_n K^{p^n}$.  Then $k$ is perfect, infinite,
  and type-definable over $B$.  Because $k$ is infinite, there is a
  $B$-indiscernible sequence $(c_i : i \in \Zz)$ in $k$, algebraically
  independent over $\Ff_p$.  By Lemma~\ref{ks-apply}, there is $N$
  such that
  \begin{equation*}
    (c_0 \cdot \wp(L \setminus K)) \cap \bigcap_{i = 1}^N (c_i \cdot
    \wp(K)) \cap \bigcap_{i = -N}^{-1} (c_i \cdot \wp(K)) =
    \varnothing.
  \end{equation*}
  Let $(a_1,\ldots,a_n) = (c_1,\ldots,c_N,c_{-1},\ldots,c_{-N},c_0)$,
  and let $\bar{a}' = \bar{a} \restriction (n-1) =
  (c_1,\ldots,c_N,c_{-1},\ldots,c_{-N})$.  Then
  \begin{equation}
    (a_n \cdot \wp(L \setminus K)) \cap \bigcap_{i = 1}^{n-1} (a_i \cdot
    \wp(K)) = \varnothing. \label{tbc}
  \end{equation}
  Because $L$ is an Artin-Schreier extension, we can write $L$ as
  $K(r)$ for some $r \in L \setminus K$ with $\wp(r) \in K$.  In particular,
  \begin{equation*}
    \text{There is $r \in \Aa(L) \setminus \Aa(K)$ with $\wp(r)
      \in \Aa(K)$.}
  \end{equation*}
  By Fact~\ref{come-on}, the projection $\pi : G_{\bar{a}} \to G_{\bar{a}'}$
  is isomorphic (in the category of algebraic groups over $k$) to the
  Artin-Schreier map $\Aa \to \Aa$.  Therefore
  \begin{equation*}
    \text{There is $\bar{s} \in G_{\bar{a}}(L) \setminus
      G_{\bar{a}}(K)$ with $\pi(\bar{s}) \in G_{\bar{a}'}(K)$.}
  \end{equation*}
  In other words, there is a tuple $(t,x_1,\ldots,x_n) \in L^{n+1}$ such that
  \begin{gather*}
    t = a_1 \cdot \wp(x_1) = a_2 \cdot \wp(x_2) = \cdots = a_n \cdot \wp(x_n) \\
    t, x_1, x_2, \ldots , x_{n-1} \in K \\
    x_n \in L \setminus K.
  \end{gather*}
  This directly contradicts Equation~(\ref{tbc}).
\end{proof}

\section{Explicit proof} \label{elementary-version}
\subsection{Additive polynomials}
Let $K$ be a field of characteristic $p$.
\begin{fact} \label{addp-facts}
  Let $G$ be a finite subgroup of $(K,+)$.  Let $f_G(x) = \prod_{a \in G}(x-a)$.
  \begin{enumerate}
  \item \label{first} $f_G$ is an additive polynomial, meaning $f_G(x+y) = f_G(x) + f_G(y)$.
  \item If $L/K$ is any extension, then $f_G$ induces an additive
    homomorphism $(L,+) \to (L,+)$ with kernel $G$.
  \item If $G = \Ff_p$, then $f_G(x) = \wp(x)$.
  \item \label{fourth} More generally, if $G = b \cdot \Ff_p$ for some $b \in
    K^\times$, then $f_G(x) = b^p \cdot \wp(x/b)$.
  \end{enumerate}
\end{fact}
Fact~\ref{addp-facts} is a fairly easy exercise.  Part (\ref{first})
is proved in \cite[Corollary~12.3(a)]{Kuhlmann}, or one can proceed by
induction on $|G|$ using the proof of the following lemma.
\begin{lemma} \label{third}
  Let $G \subseteq H$ be finite subgroups of $(K,+)$.  Then there is a
  finite subgroup $C \subseteq (K,+)$ such that $|C| = |H : G|$ and
  $f_H(x) = f_C(f_G(x))$.
\end{lemma}
\begin{proof}
  The subgroups $G$ and $H$ are finite-dimensional $\Ff_p$-linear
  subspaces of $K$.  Let $C_0$ be a complement of $G$ in $H$, meaning
  $C_0 \cap G = 0$ and $C_0 + G = H$.  Let $C = f_G(C_0)$.  Then $C$ is an additive subgroup because $f_G$ is an additive homomorphism, and $C \cong C_0$ because $C_0 \cap \ker(f_G) = C_0 \cap G =
  0$.  In
  particular, $|C| = |C_0| = |H : G|$.  Then $f_C(f_G(x))$ is a monic
  additive polynomial of degree $|C| \cdot |G| = |H|$.  We have
  $f_G(G) = 0$, so $f_C(f_G(G)) = 0$.  Similary, $f_C(f_G(C_0)) =
  f_C(C) = 0$.  So the kernel of $f_C \circ f_G$ contains both $G$
  and $C_0$, and therefore all of $H$.  The monic polynomial
  $f_C(f_G(x))$ has degree $H$ and has a zero at every element of $H$,
  so it must be $\prod_{a \in H}(x-a) = f_H(x)$.  Thus $f_H(x) =
  f_C(f_G(x))$.
\end{proof}
We will use the following two corollaries of Lemma~\ref{third}:
\begin{corollary} \label{by-one}
  If $G \subseteq H$ are finite subgroups of $(K,+)$ with $|H : G| =
  p$, then there is $b \in K^\times$ such that $f_H(x) = b^p \cdot
  \wp(f_G(x)/b)$.
\end{corollary}
\begin{proof}
  By Lemma~\ref{third} there is $C$ of order $p$ such that $f_H = f_C
  \circ f_H$.  Then $f_C(x) = b^p \cdot \wp(x/b)$ for some $b \in
  K^\times$, by Fact~\ref{addp-facts}.
\end{proof}
\begin{corollary} \label{up-cor}
  Let $G \subseteq H$ be finite subgroups of $(K,+)$.  Let $L/K$ be a
  field extension and $a \in L$ be an element.  If $f_G(a) \in K$,
  then $f_H(a) \in K$.
\end{corollary}
\begin{proof}
  By Lemma~\ref{third} there is $C \subseteq K$ with $f_H(x) =
  f_C(f_G(x))$.  Then $f_H(a) = f_C(f_G(a))$, and $f_C$ has
  coefficients in $K$.
\end{proof}
\begin{lemma} \label{down-lemma}
  Let $G_1, \ldots, G_n$ be finite subgroups of
  $(K,+)$ with $\bigcap_{i = 1}^n G_i = 0$.  Let $L/K$ be a field
  extension.  Suppose $a \in L$, and $f_{G_i}(a) \in K$ for $i = 1,
  \ldots, n$.  Then $a \in K$.
\end{lemma}
\begin{proof}
  Note that $n \ge 1$, or else the intersection would not vanish.
  Work in a monster model $\Mm \models \ACF_p$ extending $L$.  Let
  $b_i = f_{G_i}(a)$.  If $\sigma \in \Aut(\Mm/K)$, then
  $f_{G_i}(\sigma(a)) = \sigma(b) = b = f_{G_i}(a)$, so $\sigma(a) - a
  \in \ker(f_{G_i}) = G_i$.  This holds for all $i$, so $\sigma(a) - a
  \in \bigcap_i G_i = 0$, showing that $a$ is fixed by $\sigma$ and $a
  \in \dcl(K)$.  It remains to show $a \in K^{sep}$.  This holds
  because $a$ is a root of the polynomial $f_{G_1}(x) - b_1$, which
  has degree $|G_1|$, and $|G_1|$ roots (the elements of $a + G_1$).
\end{proof}
\subsection{$(L,K)$-contrary tuples}
Let $L/K$ be an Artin-Schreier extension in characteristic $p$.
\begin{definition} \label{contra-def}
  A tuple $\bar{b} \in K^n$ is \emph{$(L,K)$-contrary} if for all $1 \le i
  \le n$ we have
  \begin{equation*}
    (b_i \cdot \wp(L \setminus K)) \cap \bigcap_{j \ne i} (b_j \cdot
    \wp(K)) \ne \varnothing.
  \end{equation*}
\end{definition}
Note that Definition~\ref{contra-def} is spiritually the opposite of
Lemma~\ref{ks-apply}.
\begin{lemma} \label{contrariness}
  Let $L/K$ be an Artin-Schreier extension.
  \begin{enumerate}
  \item \label{con1} If $(b_1,\ldots,b_n)$ is $(L,K)$-contrary, so is any
    permutation $(b_{\sigma(1)},\ldots,b_{\sigma(n)})$.
  \item \label{con2} If $(b_1,\ldots,b_n)$ is $(L,K)$-contrary, so is any subtuple.
  \item \label{con3} If $L, K$ are type-definable in some structure $\Mm$, then the
    set of $(L,K)$-contrary $n$-tuples in $K$ is type-definable.
  \end{enumerate}
\end{lemma}
\begin{proof}
  Permutation invariance is clear.  By definition, $(b_1,\ldots,b_n)
  \in K^n$ is $(L,K)$-contrary if there exist $y_i$ and $x_{i,j}$ such
  that
  \begin{gather*}
    y_i = b_j \cdot \wp(x_{i,j}) \\
    x_{i,j} \in L \\
    x_{i,j} \in K \iff i \ne j.
  \end{gather*}
  If $\{y_i\}_{i \le n}$ and $\{x_{i,j}\}_{i,j \le n}$ witness that
  $(b_1,\ldots,b_n)$ is $(L,K)$-contrary, then $\{y_i\}_{i < n}$ and
  $\{x_{i,j}\}_{i,j < n}$ witness that $(b_1,\ldots,b_{n-1})$ is
  $(L,K)$-contrary.

  Finally, in the case where $K, L$ are type-definable, the set of
  $(L,K)$-contrary tuples is the projection of the type-definable set
  of triples $(\bar{b},\bar{x},\bar{y})$ satisfying the above
  conditions, which are clearly type-definable.
\end{proof}

\begin{lemma}\label{generator}
  Suppose $L/K$ is an Artin-Schreier extension and $K$ is infinite.
  For every $n \ge 2$, there is an $(L,K)$-contrary $n$-tuple in $K$.
\end{lemma}
\begin{proof}
  Let $a_1,\ldots,a_n \in K$ be $\Ff_p$-linearly independent.  Let
  $V_i$ be $\Ff_p \cdot a_i$.  Let $T = V_1 + \cdots + V_n$, the
  $\Ff_p$-linear span of $\bar{a}$.  Let $W_i$ be the $\Ff_p$-linear
  span of $\{a_1,\ldots,a_{i-1},a_{i+1},\ldots,a_n\}$.  Note $W_i
  \supseteq V_j \iff i \ne j$, and $\bigcap_{i = 1}^n W_i = 0$.
  Additionally, $f_{V_i}(x) = a_i^p \cdot \wp(x/a_i)$ by
  Fact~\ref{addp-facts}(\ref{fourth}).

  By Corollary~\ref{by-one}, there are $b_i \in K^\times$ such that
  \begin{equation}
    f_T(x) = b_i^p \cdot \wp(f_{W_i}(x)/b_i). \label{use-1}
  \end{equation}
  We claim $(b_1^p,\ldots,b_n^p)$ is $(L,K)$-contrary.  By symmetry it
  suffices to show
  \begin{equation}
    (b_1^p \cdot \wp(L \setminus K)) \cap \bigcap_{i = 2}^n (b_i^p \cdot
    \wp(K)) \stackrel{?}{\ne} \varnothing. \label{goal-1}
  \end{equation}
  As $L/K$ is an Artin-Schreier extension, $L = K(\alpha)$ for some
  $\alpha \in L \setminus K$ with $\wp(\alpha) \in K$.  The map
  $f_{V_1}(x)$ is a twist of $\wp(x)$ so we can similarly find $\beta
  \in L \setminus K$ with $f_{V_1}(\beta) \in K$.  More precisely, let
  $\beta = \alpha \cdot a_1$.  Then $\beta \in L \setminus K$, and
  \begin{equation*}
    f_{V_1}(\beta) = a_1^p \cdot \wp(\beta/a_1) = a_1^p \cdot \wp(\alpha) \in K.
  \end{equation*}
  For $i \ne 1$, we have $W_i \supseteq V_1$, and so $f_{W_i}(\beta) \in
  K$ by Corollary~\ref{up-cor}.  We claim $f_{W_1}(\beta) \in L
  \setminus K$.  Otherwise, $f_{W_i}(\beta) \in K$ for all $i$, and
  then $\beta \in K$ by Lemma~\ref{down-lemma}, a contradiction.

  By Equation (\ref{use-1}), $f_T(\beta) = b_i^p \cdot
  \wp(f_{W_i}(\beta)/b_i)$ for each $i$.  Now $f_{W_1}(\beta)/b_1 \in L \setminus
  K$ and $f_{W_i}(\beta)/b_i \in K$ for $i \ne 1$.  Thus Equation (\ref{goal-1})
  holds:
  \begin{equation*}
    f_T(\beta) \in (b_1^p \cdot \wp(L \setminus K)) \cap \bigcap_{i =
      2}^n (b_i^p \cdot \wp(K)). \qedhere
  \end{equation*}
\end{proof}

\begin{theorem-1}
  Let $\Mm$ be a monster model of an NIP theory.  Let $K$ be a
  type-definable infinite field of characteristic $p > 0$.  Then $K$
  is Artin-Schreier closed.
\end{theorem-1}
\begin{proof}
  Otherwise, take a non-trivial Artin-Schreier extension $L/K$.  Let
  $B$ be a small set of parameters defining $K$ and $L$.  By
  Lemma~\ref{contrariness}(\ref{con3}), there is a partial type
  $\Sigma(x_i : i \in \Zz)$ saying that every finite subtuple is
  $(L,K)$-contrary.  By Lemma~\ref{generator} and
  Lemma~\ref{contrariness}(\ref{con2}), $\Sigma$ is finitely
  satisfiable.  Then there is some tuple $\bar{c} = (c_i : i \in \Zz)$
  satisfying $\Sigma$.  Extracting an indiscernible sequence, we may
  assume $\bar{c}$ is $B$-indiscernible.  By Lemma~\ref{ks-apply},
  there is $N$ such that
  \begin{equation*}
    (c_0 \cdot \wp(L \setminus K)) \cap \bigcap_{i = 1}^N (c_i \cdot
    \wp(K)) \cap \bigcap_{i = -N}^{-1} (c_i \cdot \wp(K)) =
    \varnothing.
  \end{equation*}
  This contradicts the fact that $(c_{-N},\ldots,c_N)$ is
  $(L,K)$-contrary.
\end{proof}

\section{Artin-Schreier pullback hypercubes} \label{comparison}
This section contains no proofs, but is included to informally
motivate the proof in Section~\ref{elementary-version} and explain the
parallel between Sections~\ref{alg-groups-version} and
\ref{elementary-version}.

If $n$ is a positive integer, let $[n] = \{1,\ldots,n\}$, and let
$P_n$ be the poset of subsets of $[n]$, viewed as a category.  If $C$
is a category, an $n$-dimensional \emph{hypercube} in $C$ is a functor
$F : P_n \to C$.  For $A \subseteq [n]$, we write $F_A$ rather than $F(A)$.  A \emph{pullback hypercube} is a hypercube $F : P_n
\to C$ such that for any $A, B \subseteq [n]$, the following diagram
is a pullback square:
\begin{equation*}
  \xymatrix{F_A \ar[r] & F_{A \cup B} \\ F_{A \cap B} \ar[r] \ar[u] &
    F_B \ar[u].}
\end{equation*}
There are two simple ways to construct pullback hypercubes.
\begin{example}\label{direct-construction}
  Let $C$ be a category with fiber products.  Let $X$ be an object in
  $C$ and let $f_i : Y_i \to X$ be a morphism for each $i \in [n]$.
  Then there is a natural pullback hypercube $F$ in which $F_A$ is
  fiber product $Y_{i_1} \times_X Y_{i_2} \times_X \cdots \times_X
  Y_{i_k}$, where $\{i_1,\ldots,i_k\} = [n] \setminus A$.
\end{example}
\begin{example}\label{reverse-construction}
  Consider the category of abelian groups.  Say that a pullback
  hypercube is \emph{surjective} if all the connecting maps $F_A \to
  F_B$ (for $A \subseteq B$) are surjective.  Let $G$ be an abelian
  group and let $H$ be a subgroup.  Suppose $H$ is a finite direct sum
  $H_1 \oplus H_2 \oplus \cdots \oplus H_n$.  Then there is a
  surjective pullback hypercube $F$ in which $F_A$ is
  $G/\left(\bigoplus_{i \in A} H_i\right)$.  All surjective pullback
  hypercubes in the category of abelian groups arise in this way.
\end{example}
We call Example~\ref{direct-construction} the \emph{top-down}
construction, and Example~\ref{reverse-construction} the
\emph{bottom-up} construction.
\begin{remark} \label{pull-to-bs}
  Let $F$ be an $n$-dimensional pullback hypercube in the category of
  abelian groups.  Let $G_i$ be the image of $F_{[n] \setminus \{i\}}
  \to F_{[n]}$.  It is easy to see that the following are equivalent
  for each $i \in [n]$:
  \begin{enumerate}
  \item $F_{\varnothing} \to F_{\{i\}}$ is surjective.
  \item $G_i \supseteq G_1 \cap \cdots \cap G_{i-1} \cap G_{i+1} \cap
    \cdots \cap G_n$.
  \end{enumerate}
  If \emph{none} of the maps $F_{\varnothing} \to F_{\{i\}}$ is
  surjective, then
  \begin{equation*}
    G_i \not \supseteq G_1 \cap \cdots \cap G_{i-1} \cap G_{i+1} \cap
    \cdots \cap G_n
  \end{equation*}
  for all $i$.  This is a failure of a Baldwin-Saxl-like condition.
\end{remark}

Now consider the category of abelian algebraic groups over a field $K$
of positive characteristic.  An \emph{Artin-Schreier hypercube} is a
pullback hypercube $F$ satisfying the following additional conditions:
\begin{itemize}
\item $F_{A}$ is isomorphic to the additive group $\Aa$ for all $A
  \subseteq [n]$.
\item Every connective morphism $F_A \to F_B$ is geometrically
  surjective (that is, surjective on $K^{alg}$-points).
\item When $A \subseteq B$ and $|B| = |A| + 1$, the connecting map
  $F_A \to F_B$ is isomorphic to an Artin-Schreier morphism $\Aa
  \to \Aa$.
\end{itemize}
Suppose we have an Artin-Schreier hypercube over $K$, and $K$ is not
Artin-Schreier closed.  Then the Artin-Schreier map $K \to K$ is not
surjective, and therefore none of the connecting maps $F_A(K) \to
F_B(K)$ are surjective (unless $A = B$).  Let $G_i$ be the image of
$F_{[n] \setminus \{i\}}(K) \to F_{[n]}(K)$.  By
Remark~\ref{pull-to-bs}, we have
\begin{equation*}
  G_i \not \supseteq G_1 \cap \cdots \cap G_{i-1} \cap G_{i+1} \cap
  \cdots \cap G_n
\end{equation*}
for all $i$.  \emph{If} we can take $n$ arbitrarily large, then the
field $(K,+,\cdot)$ fails to be NIP by the Baldwin-Saxl lemma.

In order to prove the Kaplan-Scanlon-Wagner theorem (for example), it
suffices to find Artin-Schreier hypercubes for $n$ arbitrarily large.
There are two ways to do this---the top-down construction of
Example~\ref{direct-construction}, and the bottom-up construction of
Example~\ref{reverse-construction}.\footnote{\emph{All} pullback
hypercubes arise via the top-down construction.  The maps in an
Artin-Schreier hypercube are ``surjective'' in some sense, so every
Artin-Schreier hypercube also arises via the bottom-up construction.}

With the top-down construction, we start with algebraic groups $X,
Y_1, \ldots, Y_n$ and morphisms $f : Y_i \to X$, which will form the
top corner of the pullback hypercube.  By definition of
``Artin-Schreier hypercube'' we may as well take $X = Y_1 = \cdots =
Y_n = \Aa$, and $f_i : \Aa \to \Aa$ given by $f_i(x) = a_i \cdot
\wp(x)$, for some $a_i \in K^\times$.  Applying the top-down
construction, we get the groups $G_{\bar{a}}$ of
Definition~\ref{ga-def}.  Taking $n = 3$ for simplicity, we get this
pullback cube:
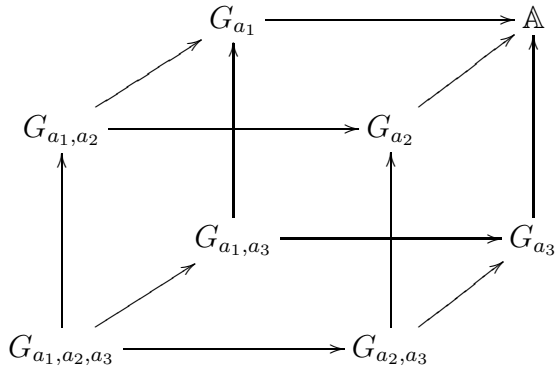
\begin{equation*}
  \xymatrix{
    & G_{a_1} \ar[rr] & & \Aa \\
    G_{a_1,a_2} \ar[ur] \ar[rr] & & G_{a_2} \ar[ur] & \\
    & G_{a_1,a_3} \ar[uu] \ar[rr] & & G_{a_3} \ar[uu] \\
    G_{a_1,a_2,a_3} \ar[uu] \ar[rr] \ar[ur] & & G_{a_2,a_3} \ar[ur] \ar[uu] &
  }
\end{equation*}
In order for this to be an Artin-Schreier hypercube, we need each
$G_{\bar{a}}$ to be isomorphic to $\Aa$.  This doesn't necessarily
hold, and the point of Corollary~2.9 in \cite{NIPfields}
(Fact~\ref{they-are-additive-group} above) is to specify a sufficient
condition on $\bar{a}$ to ensure everything works.

With the bottom-up construction, we start with the additive group
$\Aa$ and subgroups $H_1, H_2, \ldots, H_n$ such that $H_1 + \cdots
+ H_n$ is an internal direct sum.  We then define $F_A$ to be the
quotient\footnote{The quotients are taken in the category of abelian algebraic groups, which requires some caution.  If $G$ is an abelian algebraic group and $H$ is an algebraic subgroup, then $(G/H)(K)$ is not in
general isomorphic to $G(K)/H(K)$.  But we should think of
$K$-varieties via their $K^{alg}$-points, and $(G/H)(K^{alg})$
\emph{is} isomorphic to $G(K^{alg})/H(K^{alg})$, essentially because
ACF eliminates imaginaries.  When $K$ is perfect, the $K$-points of $G/H$ correspond to the $K$-definable cosets of $H(K^{alg})$ in $G(K^{alg})$.} of $\Aa$ by $\sum_{i \in A} H_i$.  In
order to get an Artin-Schreier hypercube, $H_i$ must be a finite
subgroup of order $p$, hence $\Ff_p \cdot a_i$ for some $a_i \in K$.
The requirement that $H_1 + \cdots + H_n$ is an internal direct sum
then says that $\{a_1,\ldots,a_n\}$ is $\Ff_p$-linearly independent.
Finally, $F_A$ is $\Aa$ modulo the $\Ff_p$-linear span of $\{a_i : i \in
A\}$.  Taking $n = 3$ for simplicitly, we get the following pullback
cube:
\begin{equation*}
  \xymatrix{
    & \Aa/\langle a_2, a_3 \rangle \ar[rr] & & \Aa/\langle a_1, a_2, a_3 \rangle \\
    \Aa/\langle a_3 \rangle \ar[ur] \ar[rr] & & \Aa/\langle a_1, a_3 \rangle \ar[ur] & \\
    & \Aa/\langle a_2 \rangle \ar[uu] \ar[rr] & & \Aa/\langle a_1, a_2 \rangle \ar[uu] \\
    \Aa \ar[uu] \ar[rr] \ar[ur] & & \Aa/\langle a_1 \rangle \ar[ur] \ar[uu] &
  } 
\end{equation*}
where $\langle \ldots \rangle$ denotes the $\Ff_p$-linear span of
$\{\ldots\}$.  Each of the groups $\Aa/G$ is isomorphic to $\Aa$ via
the polynomial $f_G$ of Fact~\ref{addp-facts}.  The individual arrows
are isomorphic to the Artin-Schreier map by Corollary~\ref{by-one}.


In summary, there are two different ways to construct Artin-Schreier
hypercubes.  This gives two different proofs of the
Kaplan-Scanlon-Wagner theorem on NIP fields.  The top-down
construction corresponds to the original proof in \cite{NIPfields},
while the bottom-up construction corresponds to a new proof.

This technique can be modified to give a proof of our main theorem on
type-definable NIP fields.  Because there are two ways to construct
Artin-Schreier hypercubes, we get two proofs.  The top-down
construction is exactly the proof in Section~\ref{alg-groups-version},
while the bottom-up construction is exactly the proof in
Section~\ref{elementary-version}.

The latter requires some explanation.  There is an
Artin-Schreier hypercube $F$ underlying the proof of
Lemma~\ref{generator}.  The hypercube $F$ is constructed via the
bottom-up construction, and the $V_1,\ldots, V_n$ in the proof of
Lemma~\ref{generator} correspond to the $H_1,\ldots,H_n$ in
Example~\ref{reverse-construction}.  The various objects appearing in
the proof of Lemma~\ref{generator} are explained as follows:
\begin{itemize}
\item The map $f_T$ is the map $F_{\varnothing} \to
  F_{[n]}$ on the diagonal of the cube.
\item The maps $f_{W_i}$ are the maps $F_{\varnothing} \to F_{[n]
  \setminus \{i\}}$.
\item The maps $x \mapsto b_i^p \cdot \wp(x/b_i)$ are the maps $F_{[n]
  \setminus i} \to F_{[n]}$.
\item The maps $x \mapsto a_i^p \cdot \wp(x/a_i)$ are the maps
  $F_{\varnothing} \to F_{\{i\}}$.
\item Lemma~\ref{down-lemma} is a manifestation of the fact that $F$
  is a pullback hypercube.
\end{itemize}
\begin{remark}
  Every Artin-Schreier hypercube arises via the top-down construction,
  and every Artin-Schreier hypercube arises via the bottom-up
  construction.  On some level, the two constructions are equivalent,
  and therefore so are the two proofs of the main theorem.
\end{remark}

\begin{acknowledgment}
  The author was supported by Fudan University and the National
  Natural Science Foundation of China (Grant No.\@ 12101131).  The
  author would like to thank Yatir Halevi and Itay Kaplan, who
  suggested the problem and provided helpful discussion.
\end{acknowledgment}

\bibliographystyle{plain} \bibliography{little-bib}{}

\end{document}